\documentclass{birkjour}
 \newtheorem{thm}{Theorem}[section]
 
 \newtheorem{lem}[thm]{Lemma}
 
 \theoremstyle{definition}
 \newtheorem{defn}[thm]{Definition}
 \theoremstyle{remark}
 \newtheorem{rem}[thm]{Remark}
 
 \numberwithin{equation}{section}

\begin{document}

%
%
%
%
%
%
%
%
%

\title{A fixed point theorem for Kannan-type maps in metric spaces}

\author{Mitropam Chakraborty}

\address{%
Department of Mathematics\\
Visva-Bharati\\
Santiniketan 731235\\
India}

\email{mitropam@gmail.com}

\thanks{The first author is indebted to the \textbf{UGC} (University Grants
Commissions), India for awarding him a JRF (Junior Research
Fellowship) during the tenure in which this paper was written.}

\author{S. K. Samanta}

\address{%
Department of Mathematics\\
Visva-Bharati\\
Santiniketan 731235\\
India}

\email{syamal\_123@yahoo.co.in}



\date{November 19, 2012}

\begin{abstract}
We prove a generalization of Kannan's fixed point theorem, based on
a recent result of Vittorino Pata.
\end{abstract}

\maketitle
\section{Introduction}

Our starting point is Kannan's result in metric fixed point theory involving contractive type mappings which are not necessarily continuous
\cite{key-1}. It has been shown in \cite{key-6} that Kannan's theorem is independent
of the famous Banach contraction principle (see, e.g. \cite{key-2}), and that
it also characterizes the metric completeness concept \cite{key-3}.
\begin{defn}
Let $(X,\, d)$ be a metric space. Let us call $T:\, X\rightarrow X$
a \textit{Kannan map} if there exists some $\lambda\in[0,\,1)$ such
that 
\begin{equation}
d(Tx,\, Ty)\leq\frac{\lambda}{2}\left\{ d(x,\, Tx)+d(y,\, Ty)\right\} \label{eq:a}
\end{equation}
for all $x,\, y\in X$.
\end{defn}
For complete metric spaces, Kannan proved the following:
\begin{thm}
(\cite{key-1})\label{thm:If--is} If $(X,\, d)$ is a complete metric
space, and if $T$ is a Kannan map on $X$, then there exists a unique
$x\in X$ such that $Tx=x$. 
\end{thm}
And Subrahmanyam (in \cite{key-3}) has proved the counterpart by showing
that if all the Kannan maps on a metric space have fixed points then
that space must necessarily be complete.

\section{Generalization of Kannan's fixed point theorem}

As in \cite{key-4}, from this point onwards let $(X,\, d)$ stand
for a complete metric space. Let us select arbitrarily a point $x_{0}\in X$,
and call it the "zero" of $X$. We denote
\begin{equation*}
\left\Vert x\right\Vert :=d(x,\, x_{0}), \forall x\in X.
\end{equation*}
Let $\Lambda\geq0,\,\alpha\geq1,\textrm{ and }\beta\in[0,\,\alpha]$
be fixed constants, and let $\psi:\,[0,\,1]\rightarrow[0,\,\infty)$
denote a preassigned increasing function that vanishes (with continuity)
at zero. Then, for a map $T:\, X\rightarrow X$, Pata shows
that the following theorem holds.
\begin{thm}
(\cite{key-4})\label{thm:Pata} If the inequality
\begin{equation}
d(Tx,\, Ty)\leq(1-\varepsilon)d(x,\, y)+\Lambda\varepsilon^{\alpha}\psi(\varepsilon)[1+\left\Vert x\right\Vert +\left\Vert y\right\Vert ]^{\beta}\label{eq:w}
\end{equation}
is satisfied for every $\varepsilon\in[0,\,1]$ and every $x,\, y\in X,$
then $T$ possesses a unique fixed point $x_{*}=Tx_{*}$ ($x_{*}\in X$). 
\end{thm}
Motivated by this generalization of the Banach fixed point theorem,
we can come up with an analogous generalized form of Theorem \ref{thm:If--is}.

\subsection{The main theorem}

With all the other conditions and notations remaining the same except for a more general
$\beta\geq0$, our goal is to prove the following:
\begin{thm}
\label{thm:If-the-inequality}If the inequality 
\begin{multline}
d(Tx,\, Ty)\\
\leq\frac{1-\varepsilon}{2}\left\{ d(x,\, Tx)+d(y,\, Ty)\right\} +\Lambda\varepsilon^{\alpha}\psi(\varepsilon)\left[1+\left\Vert x\right\Vert +\left\Vert Tx\right\Vert +\left\Vert y\right\Vert +\left\Vert Ty\right\Vert \right]^{\beta}\label{eq:b}
\end{multline}
is satisfied $\forall\varepsilon\in[0,\,1]$ and $\forall x,\, y\in X$,
then $T$ possesses a unique fixed point 
\[
x^{*}=Tx^{*}\,(x^{*}\in X).
\]
\end{thm}
\begin{rem}
Since we can always redefine $\Lambda$ to keep (\ref{eq:b}) valid
no matter what initial $x_{0}\in X$ we choose, we are in no way restricting
ourselves by taking that zero instead of a generic $x\in X$
\cite{key-4}. 
\end{rem}

\subsection{Proofs}

\subsubsection{Uniqueness of $x^{*}$}
We claim first that such an $x^{*}$, if it exists, is unique.
To see that this is the case, let, if possible, $\exists x^{*},\, y^{*}\in X$ such that
\begin{align*}
&x^{*}=Tx^{*},\\ 
&y^{*}=Ty^{*},\\
\textrm{and } &x^{*}\neq y^{*}.
\end{align*}
Then (\ref{eq:b}) implies, $\forall\varepsilon\in[0,\,1],$ 
\[
d(x^{*},\, y^{*})\leq\Lambda\varepsilon^{\alpha}\psi(\varepsilon)\left[1+\left\Vert x^{*}\right\Vert +\left\Vert Tx^{*}\right\Vert +\left\Vert y^{*}\right\Vert +\left\Vert Ty^{*}\right\Vert \right]^{\beta}.
\]
In particular, $\varepsilon=0$ gives us 
\begin{align*}
d(x^{*},\, y^{*}) & \leq0\\
\implies x^{*} & =y^{*},
\end{align*}
which is a contradiction.

\subsubsection{Existence of $x^{*}$}

We now bring into play the two sequences 
\begin{align*}
 & x_{n}=Tx_{n-1}=T^{n}x_{0}\\
\textrm{and } & c_{n}=\left\Vert x_{n}\right\Vert, n=1,\,2,\,3,\,\ldots.
\end{align*}
But before we proceed any further, we will need the following:
\begin{lem}
\label{lem:-is-bounded.}$\{c_{n}\}$ is bounded.\end{lem}
\begin{proof}
From (\ref{eq:b}), considering again the case of $\varepsilon=0,$
we see that for $n=1,\,2,\,3,\,\ldots$, 
\begin{align}
d(x_{n+1},\, x_{n}) & =d(Tx_{n},\, Tx_{n-1})\nonumber \\
 & \leq\frac{1}{2}\left\{ d(x_{n+1},\, x_{n})+d(x_{n},\, x_{n-1})\right\} \nonumber \\
\implies d(x_{n+1},\, x_{n}) & \leq d(x_{n},\, x_{n-1})\nonumber \\
 & \vdots\nonumber \\
 & \leq d(x_{1},\, x_{0})=c_{1}.\label{eq:c}
\end{align}

Now, $\forall n\in\mathbb{N}$,
\begin{align*}
c_{n} & =d(x_{n},\, x_{0})\\
 & \leq d(x_{n},\, x_{1})+d(x_{1},\, x_{0}) & \textrm{}\\
 & =d(x_{n},\, x_{1})+c_{1} & \textrm{}\\
 & \leq d(x_{n},\, x_{n+1})+d(x_{n+1},\, x_{1})+c_{1} & \textrm{}\\
 & \leq c_{1}+d(x_{n+1},\, x_{1})+c_{1} & \textrm{[using (\ref{eq:c})]}\\
 & \leq d(Tx_{n},\, Tx_{0})+2c_{1}\\
 & \leq\frac{1}{2}\left\{ d(x_{n+1},\, x_{n})+d(x_{1},\, x_{0})\right\} +2c_{1} & \textrm{[using (\ref{eq:b}) with }\varepsilon=0\,\textrm{]}\\
 & \leq\frac{1}{2}(c_{1}+c_{1})+2c_{1} & \textrm{[(\ref{eq:c})]}\\
 & =3c_{1}. 
\end{align*}
And hence the lemma is proved.
\end{proof}
Next we strive to show that:
\begin{lem}
\label{lem:-is-Cauchy.}$\{x_{n}\}$ is Cauchy.\end{lem}
\begin{proof}
In light of (\ref{eq:b}), for $n=1,\,2,\,3,\,\ldots$, 
\begin{align*}
&d(x_{n+1},\, x_{n})\\
&=d(Tx_{n},\, Tx_{n-1})\\
&\leq\frac{1-\varepsilon}{2}\left\{ d(x_{n+1},\, x_{n})+d(x_{n},\, x_{n-1})\right\}\\ &+\Lambda\varepsilon^{\alpha}\psi(\varepsilon)\left[1+\left\Vert x_{n+1}\right\Vert +\left\Vert x_{n}\right\Vert +\left\Vert x_{n}\right\Vert +\left\Vert x_{n-1}\right\Vert \right]^{\beta}\\
&\leq\frac{1-\varepsilon}{2}\left\{ d(x_{n+1},\, x_{n})+d(x_{n},\, x_{n-1})\right\} +C\varepsilon^{\alpha}\psi(\varepsilon)
\end{align*}
for $C=sup_{j\in\mathbb{N}}\Lambda(1+4c_{j})^{\beta}<\infty$ (on account
of Lemma \ref{lem:-is-bounded.}). But then, $\forall\varepsilon\in(0,\,1],$
\begin{align}
&d(x_{n+1},\, x_{n}) \nonumber \\
& \leq\frac{1-\varepsilon}{1+\varepsilon}d(x_{n},\, x_{n-1})+\frac{2C\varepsilon^{\alpha}}{1+\varepsilon}\psi(\varepsilon)\nonumber \\
& \vdots\nonumber \\
& \leq\cdots\nonumber \\
& \vdots\nonumber \\
& \leq\left(\frac{1-\varepsilon}{1+\varepsilon}\right)^{n}d(x_{1},\, x_{0})\nonumber \\
& +\frac{2C\varepsilon^{\alpha}}{1+\varepsilon}\psi(\varepsilon)\left[1+\frac{1-\varepsilon}{1+\varepsilon}+\cdots+\left(\frac{1-\varepsilon}{1+\varepsilon}\right)^{n-1}\right]\nonumber \\
& \leq k^{n}d(x_{1},\, x_{0})\nonumber \\
& +\frac{2C\varepsilon^{\alpha}}{1+\varepsilon}\psi(\varepsilon)(1+k+\cdots+k^{n-1}) & \textrm{[letting }k=\frac{1-\varepsilon}{1+\varepsilon}\geq0\textrm{]}\nonumber \\
& \le k^{n}d(x_{1},\, x_{0})\nonumber \\
& +\frac{2C\varepsilon^{\alpha}}{1+\varepsilon}\psi(\varepsilon)(1+k+\cdots+k^{n-1}+\cdots) & \textrm{[because }k\geq0 ]\nonumber \\
& \le k^{n}d(x_{1},\, x_{0})\nonumber \\
& +\frac{2C\varepsilon^{\alpha}}{1+\varepsilon}\psi(\varepsilon)\frac{1}{1-k}\nonumber \\
& =k^{n}d(x_{1},\, x_{0})+C\varepsilon^{\alpha-1}\psi(\varepsilon) & \textrm{[putting }k=\frac{1-\varepsilon}{1+\varepsilon}]\label{eq:z}
\end{align}
for all $n\in\mathbb{N}$. 

At this point we note that if $\varepsilon\in(0,\,1],$ then $k<1.$
Therefore, taking progressively lower values of $\varepsilon$ that
approach zero but never quite reach it, the R.H.S. of (\ref{eq:z})
can be made as small as one wishes it to be as $n\rightarrow\infty.$
Indeed, since $C\varepsilon^{\alpha-1}\psi(\varepsilon)\rightarrow0$
as $\varepsilon\rightarrow0+,$ for an arbitrary $\eta>0,$ $\exists\varepsilon=\varepsilon(\eta)>0$
such that $C\varepsilon^{\alpha-1}\psi(\varepsilon)<\frac{\eta}{2}.$
Again, for this $\varepsilon$ ($=\varepsilon(\eta)$), $\exists N\in\mathbb{N}$
such that $k^{n}d(x_{1},\, x_{0})<\frac{\eta}{2}$ $\forall n\geq N$
because $k^{n}d(x_{1},\, x_{0})\rightarrow0$ as $n\rightarrow\infty.$
Together that gives us 
\[
k^{n}d(x_{1},\, x_{0})+C\varepsilon^{\alpha-1}\psi(\varepsilon)<\frac{\eta}{2}+\frac{\eta}{2}=\eta,\forall n\geq N.
\]
In other words, 
\begin{equation}
d(x_{n},\, x_{n+1})\rightarrow0\,\textrm{as}\, n\rightarrow\infty,\,\varepsilon\rightarrow0+.\label{eq:e}
\end{equation}

Hence, from (\ref{eq:b}), using the same $C=sup_{j\in\mathbb{N}}\Lambda(1+4c_{j})^{\beta}$, and
letting $n\rightarrow\infty$, $\varepsilon\rightarrow0+,$ 
\begin{align*}
&d(x_{n},\, x_{n+p})\\
&=d(Tx_{n-1},\, Tx_{n+p-1})\\
& \leq\frac{1-\varepsilon}{2}\left\{ d(x_{n-1},\, x_{n})+d(x_{n+p-1},\, x_{n+p})\right\} +C\varepsilon^{\alpha}\psi(\varepsilon)\\
 & \rightarrow0 &\textrm{[using (\ref{eq:e})] }
\end{align*}
uniformly over $p=1,\,2,\,\ldots,$ which basically assures us that
$\{x_{n}\}$ is Cauchy. 
\end{proof}
Equipped with (\ref{lem:-is-Cauchy.}) and taking into note the
completeness of $X$, we can now safely guarantee the existence of
some $x^{*}\in X$ to which $\{x_{n}\}$ converges. 

Finally, all that remains to show is that:

\subsubsection{$x^{*}$ is a fixed point for $T.$}
For this we observe that, $\forall n\in\mathbb{N},$ 
\begin{align}
&d(Tx^{*},\, x^{*})\nonumber \\
& \leq d(Tx^{*},\, x_{n+1})+d(x_{n+1},\, x^{*})\nonumber \\
& =d(Tx^{*},\, Tx_{n})+d(x_{n+1},\, x^{*}) & \textrm{}\nonumber \\
& \leq\frac{1}{2}\left\{ d(Tx^{*},\, x^{*})+d(Tx_{n},\, x_{n})\right\}+d(x_{n+1},\, x^{*})\nonumber \\
&\textrm{[using (\ref{eq:b}) with } \varepsilon=0 \textrm{ again]}\nonumber \\
&\implies\frac{1}{2}d(Tx^{*},\, x^{*})\leq\frac{1}{2}d(x_{n},\, x_{n+1})+d(x_{n+1},\, x^{*})\label{eq:f}
\end{align}
 As $n\rightarrow\infty$ (and $\varepsilon\rightarrow0+$), we know
that: 
\begin{align*}
&d(x_{n},\, x_{n+1})\rightarrow0 &\textrm{[from (\ref{eq:e})];}\\
&d(x_{n+1},\, x^{*})\rightarrow0 &\textrm{[since }x_{n}\rightarrow x^{*}].
\end{align*} 
So (\ref{eq:f}) actually gives us that 
\begin{align*}
d(x^{*},\, Tx^{*}) & \leq0\\
\implies Tx^{*} & =x^{*},
\end{align*}
which is the required result.

\section{Comparison with Kannan's Original Result}

The requirements of Theorem \ref{thm:If-the-inequality} are indeed
weaker than those of Kannan's theorem. To see that, let us start from
(\ref{eq:a}) with $\lambda\in(0,\,1)$ (barring the trivial case where
$\lambda=0$). 

We have, $\forall\varepsilon\in[0,\,1],$ 
\begin{align}
&d(Tx,\, Ty)\nonumber \\
& \leq\frac{\lambda}{2}\left\{ d(x,\, Tx)+d(y,\, Ty)\right\} \nonumber \\
& \leq\frac{1-\varepsilon}{2}\left\{ d(x,\, Tx)+d(y,\, Ty)\right\} +\frac{\lambda+\varepsilon-1}{2}\left\{ d(x,\, Tx)+d(y,\, Ty)\right\} \nonumber \\
& =\frac{1-\varepsilon}{2}\left\{ d(x,\, Tx)+d(y,\, Ty)\right\} \nonumber \\
& +\frac{\lambda}{2}\left(1+\frac{\varepsilon-1}{\lambda}\right)\left\{ d(x,\, Tx)+d(y,\, Ty)\right\} \nonumber \\
& \leq\frac{1-\varepsilon}{2}\left\{ d(x,\, Tx)+d(y,\, Ty)\right\} +\frac{\lambda}{2}(1+\overline{\varepsilon-1})^{\frac{1}{\lambda}}\left\{ d(x,\, Tx)+d(y,\, Ty)\right\} \nonumber \\
& \textrm{[using Bernoulli's Inequality, since }\,\varepsilon-1\geq-1\,\textrm{\,\&\,\,}\frac{1}{\lambda}>1\textrm{]}\nonumber \\
& \leq\frac{1-\varepsilon}{2}\{d(x,\, Tx)+d(y,\, Ty)\}\nonumber \\
&+\frac{\lambda}{2}\varepsilon^{\frac{1}{\lambda}}\left[1+\left\{ d(x,\, x_{0})+d(x_{0},\, Tx)\right\} +\left\{ d(y,\, x_{0})+d(x_{0},\, Ty)\right\} \right]\nonumber \\
& =\frac{1-\varepsilon}{2}\{d(x,\, Tx)+d(y,\, Ty)\}+\frac{\lambda}{2}\varepsilon^{1+\gamma}[1+\left\Vert x\right\Vert +\left\Vert Tx\right\Vert +\left\Vert y\right\Vert +\left\Vert Ty\right\Vert ]\nonumber \\
& \textrm{[taking }1<\frac{1}{\lambda}=1+\gamma\,\textrm{ for some }\gamma>0]\nonumber \\
& \leq\frac{1-\varepsilon}{2}\{d(x,\, Tx)+d(y,\, Ty)\}+\frac{\lambda}{2}\varepsilon\varepsilon^{\gamma}[1+\left\Vert x\right\Vert +\left\Vert Tx\right\Vert +\left\Vert y\right\Vert +\left\Vert Ty\right\Vert ].\label{eq:g}
\end{align}
Then a quick comparison between (\ref{eq:b}) and (\ref{eq:g}) with
$\psi(\varepsilon)=\varepsilon^{\gamma}$ ($\gamma>0$) provides us
with what we need.



\end{document}